\documentclass[12pt]{amsart}
\usepackage{amsmath,amssymb,amsbsy,amsfonts,latexsym,amsopn,amstext,cite,
                                               amsxtra,euscript,amscd,bm,mathabx,mathrsfs, abraces}
\usepackage{url}
\usepackage[colorlinks,linkcolor=blue,anchorcolor=blue,citecolor=blue,backref=page]{hyperref}
\usepackage{color}
\usepackage{dsfont}
\usepackage{graphics,epsfig}
\usepackage{graphicx}
\usepackage{float} 
\usepackage[english]{babel}
\usepackage{mathtools}
\usepackage{todonotes}
\usepackage{url}
\usepackage[colorlinks,linkcolor=blue,anchorcolor=blue,citecolor=blue,backref=page]{hyperref}

\usepackage[norefs,nocites]{refcheck}

\hypersetup{breaklinks=true}

\usepackage[norefs,nocites]{refcheck}
\usepackage[english]{babel}

\newtheorem{thm}{Theorem}

\newtheorem{cor}[thm]{Corollary}

\newtheorem{lemma}[thm]{Lemma}

\newcommand{\SL}{\operatorname{SL}}

\newcommand{\sign}{\operatorname{sign}}

\numberwithin{equation}{section}
\numberwithin{thm}{section}
\numberwithin{table}{section}

\def\squareforqed{\hbox{\rlap{$\sqcap$}$\sqcup$}}
\def\qed{\ifmmode\squareforqed\else{\unskip\nobreak\hfil
\penalty50\hskip1em\null\nobreak\hfil\squareforqed
\parfillskip=0pt\finalhyphendemerits=0\endgraf}\fi}

\def\newX{Z}

\def \R {{\mathbb R}}
\def \Z {{\mathbb Z}}

\def \sfM{\mathsf M}
\def \sfE{\mathsf E}

 \def\\{\cr}
\def\({\left(}
\def\){\right)}
\def\fl#1{\left\lfloor#1\right\rfloor}

\def\mand{\qquad\mbox{and}\qquad}
\usepackage{todonotes}

\usepackage[norefs,nocites]{refcheck}

\hypersetup{breaklinks=true}

\newcommand{\bZ}{\mathbb{Z}}

\def\squareforqed{\hbox{\rlap{$\sqcap$}$\sqcup$}}
\def\qed{\ifmmode\squareforqed\else{\unskip\nobreak\hfil
\penalty50\hskip1em\null\nobreak\hfil\squareforqed
\parfillskip=0pt\finalhyphendemerits=0\endgraf}\fi}

\def \R {{\mathbb R}}
\def \Z {{\mathbb Z}}

 \def\\{\cr}
\def\({\left(}
\def\){\right)}
\def\fl#1{\left\lfloor#1\right\rfloor}

\def\mand{\qquad\mbox{and}\qquad}

 \numberwithin{dummy}{section}

\begin{document}

\title[Counting elements of the congruence subgroup]{Counting elements of the congruence subgroup}

\author[K. Bulinski] {Kamil Bulinski}
\address{School of Mathematics and Statistics, University of New South Wales, Sydney NSW 2052, Australia}
\email{k.bulinski@unsw.edu.au}


\author[I. E. Shparlinski] {Igor E. Shparlinski}
\address{School of Mathematics and Statistics, University of New South Wales, Sydney NSW 2052, Australia}
\email{igor.shparlinski@unsw.edu.au}

\begin{abstract}  We obtain asymptotic formulas for the number of matrices in the 
 congruence subgroup 
\[
\Gamma_0(Q) = \left\{ A\in\SL_2(\Z):~c \equiv 0 \pmod Q\right\}, 
\]
which are of naive height at most $X$. Our result is uniform in a
very broad range of values $Q$ and $X$.
 \end{abstract}



\maketitle

\tableofcontents

\section{Introduction and the main result}
Given an integer $Q\ge 1$ we consider the congruence subgroup 
\[
\Gamma_0(Q) = \left\{  
A\in\SL_2(\Z):~c \equiv 0 \pmod Q\right\}, 
\]
where 
\[
A =  \begin{bmatrix}
    a   & b \\
    c  & d \\
\end{bmatrix}. 
\]
We are interested in counting matrices $A\in  \Gamma_0(Q) $ with entries 
of  size  at most 
\begin{equation}
\label{eq: Size A L_inf}
\|A\|_\infty = \max\{|a|, |b|, |c|, |d|\} \le X.
\end{equation}
The question is a natural generalisation of the  a classical counting result of Newman~\cite{New} concerning matrices $A \in \SL_2(Z)$ with 
\begin{equation}
\label{eq: Size A L_2}
\|A\|_2 = a^2 + b^2 + c^2 + d^2\le X.
\end{equation}
We note that while we can also use the $L^2$-norm as~\eqref{eq: Size A L_2} to measure the 
``size'' of $A\in \SL_2(\Z)$, for us it is more convenient to use the $L^\infty$-norm as~\eqref{eq: Size A L_inf}. 

Let 
\[
\Gamma_0(Q,X)= \{A \in \Gamma_0(Q):~\|A\|_\infty  \le X\}.
\]
The question of investigating the cardinality $\# \Gamma_0(Q, X)$ has been 
risen in~\cite{BOS},  where it is also shown  that for $Q\le X$  we have 
\[
\# \Gamma_0(Q, X) = X^{2+o(1)} Q^{-1}.
\]

We are interested in obtaining an asymptotic formula for the cardinality
$\# \Gamma_0(Q,X)$ in a broad range of $Q$ and $X$.

Here we use some results 
 of Ustinov~\cite{Ust1} to obtain the following 
 asymptotic formula, where as usual $\varphi(k)$ denotes the Euler function.
 
 We first give an asymptotic formula for $\# \Gamma_0(Q, X) $
 with the main term expressed via sums of some standard arithmetic functions. 
 For this we also define
\[
 F(X,Q)   = 8 \( F_1(Q,X) + F_2(Q,X)\), 
\]
 where
 \begin{align*}
F_1(Q,X) & =  \sum_{1 \le c \le X/Q} \frac{\varphi(cQ)}{cQ},  \\
F_2(Q,X) & = Q^{-1}    \sum_{\substack{Q< x \le X\\ \gcd(x,  Q) = 1}}  \frac{\varphi(x)}{x} . 
 \end{align*}

\begin{thm}
\label{thm:Card G}
Uniformly over an integer $Q\ge 1$ and  a positive real $X\ge Q$, we have 
\[
\# \Gamma_0(Q, X) =  X F(X,Q)  +  O\(X^{5/3+o(1)} Q^{-1} + X^{1+o(1)}\).
\] 
\end{thm}

Next we study the function $F(X,Q)$. 
We first recall the definition of the Dedekind function
\[
\psi(Q) = Q \prod_{\substack{p\mid Q\\p~\mathrm{prime}}} \(1+\frac{1}{p}\).
\]

\begin{thm}  
\label{thm:Func F}
Uniformly over an integer $Q\ge 1$ and  a positive real $X\ge Q$, we have
\[
F(Q, X) =   \frac{96}{\pi^2} \cdot \frac{X}{ \psi(Q)} + O(Q^{-1} \log X).
\] 
\end{thm}

Combining Theorems~\ref{thm:Card G}  and~\ref{thm:Func F}, we obtain.

\begin{cor}  
\label{cor:AsymFormula}
Uniformly over an integer $Q\ge 1$ and  a positive real $X\ge Q$, 
\[
\# \Gamma_0(Q, X) =   \frac{96}{\pi^2} \cdot \frac{X^2}{ \psi(Q)} +  O\(X^{5/3+o(1)} Q^{-1} + X^{1+o(1)}\).
\] 
\end{cor}

We remark the appearance of the  Dedekind function $\psi(Q)$ in the denominator of the  
asymptotic formula for $\# \Gamma_0(Q, X)$ in  Corollary~\ref{cor:AsymFormula} is not surprising as function itself appears in as  the index of $\Gamma_0(Q)$ in $\SL_2(\Z)$,  that is 
\[
[\SL_2(\Z):  \Gamma_0(Q)] =  \psi(Q), 
\]
see~\cite[Proposition~2.5]{Iwan}. 

Elementary estimates easily show that $\psi(Q) = Q^{1+o(1)}$. Thus Corollary~\ref{cor:AsymFormula} is 
nontrivial in an essentially full range of $Q$ and $X$, namely for $Q \le X^{1-\varepsilon}$ for a fixed
$\varepsilon > 0$.

\section{Preparations}  

\subsection{Notation and some elementary estimates}  
We recall that  the notations $U = O(V)$, $U \ll V$ and $ V\gg U$  
are equivalent to $|U|\leqslant c V$ for some positive constant $c$, 
which throughout this work, are absolute.

 Futhermore we write $U \asymp V$ to express that  $V \ll U \ll V$. 

We also write $f(X) = o(g(X))$ if for all $\varepsilon>0$ there exists $X_{\varepsilon}>0$ so that $|f(X)| \leq \varepsilon |g(X)|$ for all $X > X_{\varepsilon}$.

The letter $p$ always denotes a prime number. 

For an integer $k\ne 0$ we denote by $\mu(k)$,  $\tau(k)$ and $\varphi(k)$, the M{\"o}bius function,   the number of integer positive divisors  and the Euler function of $k$,  respectively, 
for which we   use the well-known bound
\begin{equation}
\label{eq:tau}
\tau(k) = |k|^{o(1)}  \mand  
  \varphi(k) \gg \frac{k}{\log \log (k+3)}, 
\end{equation}
as $|k| \to \infty$, see~\cite[Theorems~317 and~328]{HaWr}.

As usual we define
\[
\sign u = \begin{cases} 
-1, & \text{if } u < 0,\\
0, & \text{if } u =0,\\
1, & \text{if } u >0.
\end{cases}
\]

 For positive integers $u$ and $v$,  using the M\"obius function $\mu(e)$ and the inclusion-exclusion principle to detect the co-primality
condition and then interchanging the order of summation, we obtain 
 \begin{equation}
\label{eq:phi uv}
\begin{split}
  \sum_{\substack{1\le  c \le v\\ \gcd(c,u) = 1}} 1
& =\sum_{e\mid u}\mu (e) \fl{\frac{v}{e}} =  v
\sum_{e\mid u}\frac{\mu (e)}{e} +O\( \sum_{e\mid u}|\mu (e)|\)\\
& =  v \frac{\varphi(u)}{u} +O\( \tau(u)\) =     v \frac{\varphi(u)}{u} +O\(u^{o(1)}\), 
\end{split}
\end{equation} 
see~\cite[Equation~(16.1.3)]{HaWr}.

\subsection{Modular hyperbolas} 
\label{sec:ModHyp}
Here we need  some results on the distribution of points 
on the modular hyperbola
\begin{equation}
\label{eq: ModHyp}
uv \equiv 1 \pmod q, 
\end{equation}
where $q \ge 1$ is an arbitrary integer.

We start with a very well-known case counting the number $N(q;U,V)$
of 
solutions in a rectangular domain $(u,v) \in [1, U]\times [1,V]$. 
For example, such a result has been recorded in~\cite[Theorem~13]{Shp}
(we note that the restriction $U,V\le q$ is not really necessary.

\begin{lemma} 
\label{lem:ModHyp-Box}  For any $U,V\ge 1$, we have 
\[
N(q;U,V)= \frac{\varphi(q)}{q^2}  UV + O\(q^{1/2+ o(1)}\).
\]
 \end{lemma} 

Next  we recall a result of   Ustinov~\cite{Ust1} on 
 the number $T_f(q; \newX,U)$ of points $(u,v)$ on the modular hyperbola~\eqref{eq: ModHyp} 
with  variables run through a domain 
of the form 
\[
\newX<  u \le \newX+U \mand 0 \le v \le f(u), 
\]
where $f$ is a positive function with a continuous second derivative.

Namely a special case of~\cite{Ust1}, where we have also  used~\eqref{eq:tau} to estimate various divisor sums, 
can be formulated as follows.

Let  
\begin{align*}
\mathcal{T}_f(q,\newX,U) = \{(u,v) \in \bZ^2 :~  \newX < u \le \newX+U, \   0  &< v \le f(u), \\ & uv  \equiv 1 \pmod q \}
\end{align*}
and let 
\[
T_f(q,\newX,U) = \# \mathcal{T}_f(q,\newX,U) .
\]

\begin{lemma} 
\label{lem:ModHyp-Curve} Assume that the function $f:\R\to \R_{\geq 0}$ has a continuous second derivative
on $[\newX,\newX+U]$ such that for some $L >0 $ we have 
\[
 \left|f''(u)\right| \asymp \frac{1}{L}, \qquad u \in [\newX,\newX+U].
\]
Then we have the estimate
\[
T_f(q; \newX,U) = \frac{2}{q} \sum_{\substack{\newX < u \le \newX + U\\ \gcd(u,q) = 1}}  f(u)+ O\(\(U L^{-1/3} + L^{1/2} + q^{1/2}\)(qU)^{o(1)}\).
\]  
 \end{lemma} 
 
  For other results on the distribution of points on 
modular hyperbolas we refer to the survey~\cite{Shp} and also more recent works~\cite{Baier,BrHay,Chan,GarShp,Hump,Ust2}.

\section{Proof of Theorem~\ref{thm:Card G}}  

\subsection{Separating contributions to the main term and to the error term} 
\label{sec: separating}
It is convenient to assume that $X$ is integer. 

It is easy to see that there are only $O(X)$ matrices  in  $\SL_2(\Z; X)$ 
with $abcd=0$. We now consider the following eight sets for different choices of the signs 
of $a$, $c$ and $d$:
\[
\Gamma_0^{\alpha,  \gamma, \delta}(Q,X)=
\{A\in \Gamma_0(Q,X):~ \sign a =  \alpha, \ \sign c = \gamma,\ \sign d = \delta\}, 
\]
with  $\alpha,  \gamma, \delta \in \{-1,1\}$.

Now observe that $\Gamma_0(Q,X)$ is preserved under the bijections
\[\begin{bmatrix}
    a   & b \\
    c  & d \\
\end{bmatrix} \mapsto \begin{bmatrix}
    -a   & b \\
    c  &- d \\
\end{bmatrix}
\]
and 
\[\begin{bmatrix}
    a   & b \\
    c  & d \\
\end{bmatrix} \mapsto \begin{bmatrix}
    a   & -b \\
    -c  & d \\
\end{bmatrix}. 
\] This means
\[
\# \Gamma_0^{1,1,1}(X,Q) = \# \Gamma_0^{\alpha, \gamma, \alpha}
\] and 
\[\# \Gamma_0^{1,1,-1}(X,Q) = \# \Gamma_0^{\alpha, \gamma, -\alpha}
\] 
for all pairs $\alpha,\gamma \in \{-1,1\}$.

Thus
\begin{equation}
\label{eq: G and Gabc}
\# \Gamma_0(Q,X) =  4 \cdot (\# \Gamma_0^{1,  1, 1 }(Q,X) + \# \Gamma_0^{1,  1, -1 }(Q,X)) + O(X).
\end{equation}

\subsection{Preliminary counting $\Gamma_0^{1,  1, 1 }(Q,X)$} 
Writing $cQ$ instead of $c$, we need to count the number of solutions to the equation 
\[
ad = 1 + bcQ, \qquad 1 \le a, |b|,d\le X, \ 1 \le c \le X/Q.
\]
We first do this for a fixed $c$ and then sum up over all $c \le X/Q$.  

First we consider the values $a \le cQ$.
For this, we note that setting 
\[
b = \frac{ad -1}{cQ} \le X
\]
for a solution $(a,d)$ to the congruence 
\[
ad \equiv  1 \pmod {cQ} \qquad 1 \le a \le cQ, \ 1 \le d\le X.
\]
we have $b \le X$. Hence, we see from Lemma~\ref{lem:ModHyp-Box} 
(and then recalling that $cQ  \le X$) that
for every $c \in [1, X/Q]$ there are  
 \begin{equation}
\label{eq: G1c}
\begin{split}
G_{1}(c) & = \frac{\varphi(cQ)}{(cQ)^2} cQ X  + O\(\(cQ\)^{1/2 + o(1)}\)\\
& =  \frac{\varphi(cQ)}{cQ} X  + O\(X^{1/2 + o(1)}\)
\end{split}
\end{equation}  
such matrices 
\[
  \begin{bmatrix}
    a   & b \\
    cQ  & d \\
\end{bmatrix} \in \Gamma_0^{1,1,1,}(Q,X).
\]

Next we count the contribution $G_{2}(c)$ from  matrices $A\in  \Gamma_0^{1,1,1,}(Q,X)$ with $a > cQ$. To do this, we recall the notation of Section~\ref{sec:ModHyp} and then  parametrise this set using a modular hyperbola as follows.

\begin{lemma}\label{Lemma: bijection} Fix $1 \leq c \leq X/Q$, $0 < U \leq X - cQ$ and define 
\[f_c(x) = \frac{cQX+1}{x}.\] 
Then the map \[\mathcal{T}(f_c, cQ, U)  \to \Gamma_0^{1,1,1}(Q,X)\] given by 
\[
(x,y) \mapsto   \begin{bmatrix}
    x  & (xy-1)/cQ \\
    cQ  & y \\
\end{bmatrix}\]
is well defined, injective and its image is exactly the set of those $A\in  \Gamma_0^{1,1,1,}(Q,X)$ with $cQ < a \leq cQ + U$ and bottom left entry equal to $cQ$.
\end{lemma}

\begin{proof} For $(x,y) \in \mathcal{T}(f_c,cQ,U)$ we have that $(xy-1)/cQ \in \bZ$ and 
\[
0 < y \leq f_c(X),
\] which is equivalent to 
\[\frac{-1}{cQ} < (xy - 1)/cQ \leq X.
\] 
As $x > cQ \geq 1$ and $y>0$ this is actually equivalent to 
\[1 \leq (xy - 1)/cQ \leq X.\] 
We also need to check that $1 \leq y \leq X$. This follows since 
\[0<y \leq f_c(x) = \frac{cQX+1}{x} < \frac{cQX+1}{cQ} = X+\frac{1}{cQ} \leq X+1.\] Thus indeed $(x,y)$ is mapped to an element of $\Gamma_0^{1,1,1}(Q,X)$ with the desired properties. Conversely, suppose that $A \in \Gamma_0^{1,1,1}(Q,X)$ with $a>cQ$ and bottom left entry equal to $cQ$. As $ad \equiv 1 \pmod cQ$ we have $1 \leq x,y \leq X$ such that
\[ A = \begin{bmatrix}
    x  & (xy-1)/cQ \\
    cQ  & y \\
\end{bmatrix}.\]

Also by definition (the lower bound holds as $x>cQ\geq 1$) 
\[1 \leq \frac{xy - 1}{cQ} \leq X,\] 
which means 
\[0 < \frac{cQ+1}{x} \leq y \leq \frac{cQX + 1}{x} = f_c(x)\]
and so indeed $(x,y) \in \mathcal{T}(f_c, cQ, U)$. \end{proof}

We partition the interval $(cQ,X]$ into $I \ll \log X$ dyadic intervals of the form $(\newX_i,  \newX_i+U_i]$ with 
\[\newX_i =  2^{i-1} cQ \mand U_i \le \newX_i, \qquad i =1, \ldots, I, 
\]
(in fact $U_i = \newX_i$, except maybe for  $i=I$) 
and note that 
\begin{equation}
\label{eq: cQ2I}
 2^{I} cQ \asymp X.
\end{equation}
 
We now write 
\begin{equation}
\label{eq: G2 and Tf} 
G_{2}(c) = \sum_{i=1}^I T_{f_c}\(cQ; \newX_i, U_i \), 
\end{equation}
where $f_c(x)$ is as in Lemma~\ref{Lemma: bijection}.

Next, for each  $i =1, \ldots, I$, we use Lemma~\ref{lem:ModHyp-Curve} with $q = cQ$  and use that 
\[
 \left|f''(x)\right|\asymp   \frac{cQ X}{\newX_i^3}  \asymp    \frac {X}{ 2^{3i} (cQ)^2}
\]
for $x \in (Z_i,  Z_i+U_i]$.  Therefore, we conclude that 
 \begin{equation}
\label{eq: Ti Mi}
T_{f_c}\(cQ; \newX_i , U_i\)  = M_{i}(c) + O\(E_{i}(c) X^{o(1)}\), 
\end{equation}  
where
\begin{align*}
&M_{i}(c)  =  \frac{1}{cQ}  \sum_{\substack{\newX_i < x \le \newX_i+ U_i\\ \gcd(x, cQ) = 1}}  f_c (x),\\
& E_{i}(c)=2^{i} cQ \(  \frac {X}{ 2^{3i} (cQ)^2} \)^{1/3}
 +  \( \frac{ 2^{3i} (cQ)^2}{X} \)^{1/2} +X^{1/2}. 
\end{align*}
Combing the main terms $M_{i}(c)$, $i =1, \ldots, I$,  together and recalling~\eqref{eq: G2 and Tf}, 
we obtain 
 \begin{equation}
\label{eq: G2c prelim}
G_{2}(c) = \sfM(c) +  O\(\sfE(c)   X^{o(1)}\), 
\end{equation}  
where
\[
 \sfM(c)  =  \frac{1}{ cQ}  \sum_{\substack{cQ< x \le X\\ \gcd(x, |c|Q) = 1}}  f_c (x)
 \]
 and 
\begin{align*}
\sfE(c)  &  = \sum_{i=1}^I 
\(2^{i} cQ \(  \frac {X}{ 2^{3i} (cQ)^2} \)^{1/3}
 +  \( \frac{ 2^{3i} (cQ)^2}{X} \)^{1/2} + \(cQ\)^{1/2}\)\\
&  = \sum_{i=1}^I 
\( \(cQX\)^{1/3} + 2^{3i/2}  cQ X^{-1/2} + \(cQ\)^{1/2}\)\\
&  = \( \(cQX\)^{1/3} + 2^{3I/2}  cQ X^{-1/2} + \(cQ\)^{1/2}\)X^{o(1)}.
\end{align*}
Recalling~\eqref{eq: cQ2I} and using $cQ \le X$ we obtain 
\[
\sfE(c)    \le \(X^{2/3} +   (cQ)^{-1/2}  X\) X^{o(1)}.
\]
which after the substitution in~\eqref{eq: G2c prelim} yields
 \begin{equation}
\label{eq: G2c}
G_{2}(c) = \sfM(c) +  O\( \(X^{2/3} +   (cQ)^{-1/2}  X\)  X^{o(1)}\).
\end{equation}  

\subsection{Asymptotic  formula for $\Gamma_0^{1,  1, 1 }(Q,X)$} 
From the equations~\eqref{eq: G1c} and~\eqref{eq: G2c}  we obtain 
 \begin{equation}
\label{eq: G111 M E}
\# \Gamma_0^{1,1,1}(Q,X)  = \sum_{1 \le c \le X/Q} \(G_1(c) + G_2(c)\)
=  \mathbf{M} + O\(\mathbf{E}\), 
\end{equation}
where 
\begin{align*}
 \mathbf{M} & = \sum_{1 \le c \le X/Q}  \( \frac{\varphi(cQ)}{cQ} X  + 
 \frac{1}{ cQ}  \sum_{\substack{cQ< x \le X\\ \gcd(x,  cQ) = 1}}  f_c (x)\)\\
 & = X F_1(Q,X) +  \sum_{1 \le c \le X/Q} \frac{1}{ cQ}  \sum_{\substack{cQ< x \le X\\ \gcd(x,  cQ) = 1}}  f_c (x)
 \end{align*}
and 
\[
 \mathbf{E} = \sum_{1 \le c \le X/Q}    \(X^{2/3} +   (cQ)^{-1/2}  X\)  X^{o(1)}
 =   X^{5/3+o(1)} Q^{-1}   . 
\]

We also note that 
\begin{align*}
 \frac{1}{ cQ}  \sum_{\substack{cQ< x \le X\\ \gcd(x,  cQ) = 1}}  f_c (x) &
 =   \sum_{1 \le c \le X/Q}  \sum_{\substack{cQ< x \le X\\ \gcd(x,  cQ) = 1}}  
 \frac{cQ X + 1}{cQx}  \\
& =  X \sum_{1 \le c \le X/Q}  \sum_{\substack{cQ< x \le X\\ \gcd(x,  cQ) = 1}}  
  \frac{1}{x}  + O\(X^{o(1)}\).
 \end{align*}

Change the order of summation, we write 
\[
 \sum_{1 \le c \le X/Q}  \sum_{\substack{cQ< x \le X\\ \gcd(x,  cQ) = 1}}  
  \frac{1}{x}    =   \sum_{\substack{Q< x \le X\\ \gcd(x,  Q) = 1}}  \frac{1}{x}
 \sum_{\substack{c< x/Q\\ \gcd(x,  c) = 1}} 1.
\]
Hence, recalling~\eqref{eq:phi uv},  we derive that 
\begin{align*}
 \sum_{1 \le c \le X/Q}  
 \frac{1}{ cQ}  \sum_{\substack{cQ< x \le X\\ \gcd(x,  cQ) = 1}}  f_c (x)  
& =  Q^{-1}    \sum_{\substack{Q< x \le X\\ \gcd(x,  Q) = 1}}  \frac{\varphi(x)}{x}    +O\(X^{o(1)}\)\\
&  = F_2(Q,X)   +  O\(X^{o(1)}\). 
 \end{align*}

 Thus, we see from~\eqref{eq: G111 M E} that 
\begin{equation}
\label{eq:  Asymp G111} 
\# \Gamma_0^{1,1,1}(Q,X)  = X\( F_1(Q,X) + F_2(Q,X)\) + O\(X^{5/3+o(1)} Q^{-1}\). 
\end{equation} 
%

\subsection{Counting $\Gamma^{-1,1,1}(Q,X) $} Recalling~\eqref{eq: G and Gabc} 
 we see that it remains to count $\Gamma_0^{-1,  1, 1 }(Q,X)$.  One can use a similar 
 argument, but In fact we show that
\begin{equation}
\label{eq: negative values same count} 
\# \Gamma^{-1,1,1}(Q,X) =\# \Gamma^{1,1,1}(Q,X) + O\(\mathbf{E} + X\),
\end{equation} 
where the error term $\mathbf{E} = O(X^{5/3+o(1)} Q^{-1}  )$ is the same as obtained above.

Thus we wish to count matrices of the form 
\[A = \begin{bmatrix}
    x  & (xy-1)/cQ \\
    cQ  & y \\
\end{bmatrix}, 
\]
where $xy \equiv 1 \pmod {cQ}$, $-X\leq x \leq -1$, $1 \leq y \leq X$, $1 \leq cQ \leq X$ and $-X \leq (xy-1)/cQ \leq -1$. 

Without loss of generality we can assume that $X \not \in \Z$. 
Then we consider the following two cases.

\textbf{Case~I: $x> -cQ$.}   Note that for any $x,y$ with $xy \equiv 1 \pmod{cQ}$, $-cQ < x \leq -1$ and $1 \leq y \leq X$ we have 
\[ 
\frac{-cQX - 1}{cQ} < \frac{xy-1}{cQ} \leq \frac{-2}{cQ},
\] 
and so 
\[-X \leq \frac{xy-1}{cQ} \leq -1.
\] Thus indeed the corresponding $A$ is in $\Gamma_0^{-1,1,1}(Q,X)$. Note that since $0<x + cQ \leq cQ$ and $-X \leq (xy-1)/cQ + y \leq X $ we have that 
\[\begin{bmatrix}
    1 & 1\\
    0  & 1 \\
\end{bmatrix} A = \begin{bmatrix}
    x + cQ  & (xy-1)/cQ + y \\
    cQ  & y \\
\end{bmatrix} \in \Gamma_0^{1,1,1}(Q,X).  
\] 
So in fact the number of such matrices $A$ is 
exactly $G_1(c)$ as computed in~\eqref{eq: G1c} in the $\Gamma_0^{1,1,1}(Q,X)$ case.


\textbf{Case~II: $-X<x \leq -cQ$.}  Let 
\[
\widetilde{f}_c(x) = \frac{-cQX + 1}{x}.
\] 
We now need an analogue of Lemma~\ref{Lemma: bijection}. 

\begin{lemma}\label{lem: bijection negative case} Fix $1 \leq c \leq X/Q$, $0 < U \leq X - cQ$.
Then the map \[\mathcal{T}_{\widetilde{f}_c}(cQ, -X, U)  \to \Gamma_0^{-1,1,1}(Q,X)\] given by 
\[
(x,y) \mapsto  A= \begin{bmatrix}
    x  & (xy-1)/cQ \\
    cQ  & y \\
\end{bmatrix}\]
is well defined, injective and its image is exactly the set of those $A\in  \Gamma_0^{-1,1,1,}(Q,X)$ with $-X < x \leq -X + U$ and bottom left entry equal to $cQ$.
\end{lemma}  

\begin{proof} Let $(x,y) \in \mathcal{T}_{\widetilde{f}_c}(cQ, -X, U) $. Thus by definition 
\[0<y \leq \frac{-cQX+1}{x}.\]
 As $x<-cQ$ we have that 
 \[\frac{-cQX+1}{x} = \frac{cQX-1}{-x} \leq \frac{cQX-1}{cQ} < X\]
 and so indeed $y \leq X$. Moreover, as $x<0$ we have 
\[
1 \geq \frac{1}{cQ}> \frac{xy - 1}{cQ} \geq -X.
\] 
So indeed this mapping has range inside $\Gamma_0^{-1,1,1,}(Q,X)$. Conversely suppose 
\[A = \begin{bmatrix}
    x  & b \\
    cQ  & y \\
\end{bmatrix}
\] is in $\Gamma_0^{-1,1,1,}(Q,X)$ with $-X <x \leq -X + U$. 
Then $-X \leq b \leq 0$ is an integer thus $xy \equiv 1 \pmod cQ$ and 
\[-X \leq \frac{xy-1}{cQ} \leq 0.\]
Thus as $x<0$ we have 
\[\frac{-cQX + 1}{x} \geq y.  \]
 Thus 
 \[0 < y \leq \widetilde{f}_c(x)\]
 and so indeed $(x,y) \in \mathcal{T}_{\widetilde{f}_c}(cQ, -X, U)$ as desired.\end{proof}

We now fix $c$ with $1 \leq c \leq X/Q$ and observe now that 
by Lemma~\ref{lem: bijection negative case}, for any $Z \in [-X,0)$ and $0 < U \leq |Z|$ we have that $T_{\widetilde{f}_c}(cQ, Z, U)$ has the main term 
\begin{align*} \frac{2}{cQ} \sum_{\substack{\newX < x \le \newX + U\\ \gcd(x,q) = 1}} \widetilde{f}_c(x) &= 
\frac{2}{cQ} \sum_{\substack{\newX < x \le \newX + U\\ \gcd(x,q) = 1}} \frac{-cQX + 1}{x} \\
&=\frac{2}{cQ} \sum_{\substack{-\newX - U \leq x < -\newX \\ \gcd(x',q) = 1}} \frac{cQX -  1}{x} \\ 
&= \frac{2}{cQ} \sum_{\substack{-\newX - U \leq x < -\newX \\ \gcd(x,q) = 1}} f_c(x) \\
&= \frac{2}{cQ} \sum_{\substack{|\newX| - U \leq x < |\newX| \\ \gcd(x,q) = 1}} f_c(x),    \end{align*}
where we recall $f_c(x) = (cQX-1)/x$ as used in Lemma~\ref{Lemma: bijection}. But this is precisely the same main term as for $T_{f_c}(cQ, |\newX| - U, U)$ except for the boundary terms ($x = -Z-U, -Z$) which  contribute only   $O(X)$ (uniformly in $Q$ as $|(cQX-1)/x \leq (cQX-1)/cQ \leq X$). 
Thus, recalling~\eqref{eq: G2 and Tf}, \eqref{eq: Ti Mi} and~\eqref{eq: G2c prelim}, we see that for each admissible $c$, we obtain the  contribution to  $\# \Gamma^{-1,1,1}(Q,X)$,  
which is asymptotic to $G_2(c)$. Now observe that $\widetilde{f}_c(x) = - f_c(x)$ and so $|\widetilde{f}''_c(x)| = |f''_c(-x)|$ which means that the error terms we obtain from applying Lemma~\ref{lem:ModHyp-Curve} to $\widetilde{f}_c$ are the same as those obtained for $f_c$ (we have $x \in [-X, -cQ]$ and before we had $x \in [cQ, X]$). Thus if we sum over $c$ and proceed as before, we see that the error term is at most $O\(\mathbf{E} + X\)$ which implies~\eqref{eq: negative values same count}. 

\subsection{Concluding the proof} 
Substituting~\eqref{eq: negative values same count} in~\eqref{eq: G and Gabc} 
implies
\[
\# \Gamma_0(Q,X) =  8\# \Gamma_0^{1,  1, 1 }(Q,X)   +O(X^{5/3+o(1)} Q^{-1} + X)  .
\]
Recalling~\eqref{eq:  Asymp G111}, we conclude the proof. 

 \section{Proof of Theorem~\ref{thm:Func F}}

\subsection{Approximating $F_1(Q,X)$}
 
 For convenience we let
\[
G(Q,X)  =   \sum_{1 \le n \le X}  \frac{\varphi(Qn)}{Qn}.
\]
So 
\begin{equation}
\label{eq: G and F}
F_1(Q,X) = G(Q,Q^{-1}X).
\end{equation}

We now define the function
\[
h(n) = \mu(n)/n.
\]

\begin{lemma} 
\label{lem:G-expand} 
We have 
\[
G(Q,X) =\frac{\varphi(Q)}{Q} \sum_{\substack{n \leq X\\ \gcd(n,Q) = 1}} h(n) \left \lfloor \frac{X}{n} \right\rfloor.
\]
\end{lemma}

\begin{proof} Observe that for any integer $n \ge 1$, 
\[
\varphi(Qn) =  Qn \prod_{p\mid Qn} \(1-p^{-1}\)
\mand 
\varphi(Q)n =  Qn \prod_{p\mid Q} \(1-p^{-1}\). 
\]
Hence
\[
 \frac{\varphi(Qn)}{\varphi(Q)n}  =  \prod_{\substack{p \mid n\\ \gcd(p,Q) = 1}} \(1 - p^{-1}\).
\]
Thus we derive 
\begin{align*}
 \frac{Q}{\varphi(Q)} G(Q,X) 
&= \sum_{n \leq X} \prod_{\substack{p  \mid n\\ \gcd(p,Q) = 1}} \(1 - p^{-1}\) 
= \sum_{n \leq X} \sum_{\substack{d  \mid n\\ \gcd(d,Q) = 1}} \frac{\mu(d)}{d} \\
&= \sum_{\substack{d \leq X\\ \gcd(d,Q) = 1}} \sum_{\substack{n \leq X\\ d \mid n}} \frac{\mu(d)}{d} = \sum_{\substack{d \leq X\\ \gcd(d,Q) = 1}} \frac{\mu(d)}{d} \left\lfloor \frac{X}{d} \right\rfloor, 
\end{align*}
which completes the proof.\end{proof}

We now  see from Lemma~\ref{lem:G-expand}  that 
\begin{align*} G(Q,X) &= \frac{\varphi(Q)}{Q} X \sum_{\substack{n \leq X\\ \gcd(d,Q) = 1}} \frac{h(n)}{n}  + O\left(\frac{\varphi(Q)}{Q}  \sum_{n \leq X} h(n) \right) \\ 
&= \frac{\varphi(Q)}{Q} X \sum_{\substack{n \leq X\\ \gcd(n,Q) = 1}} \frac{h(n)}{n} +  O\left( \ \frac{\varphi(Q)}{Q} \log X \right) .\end{align*}

Using that 
\[
  \sum_{n>X} \frac{|h(n)|}{n} \le   \sum_{n>X} \frac{1}{n^2}= O\(X^{-1}\),
\]
we write 
\begin{equation}
\label{eq: G(Q,X) and h}
G(Q,X) = \frac{\varphi(Q)}{Q} X \sum_{\substack{n = 1\\ \gcd(n,Q) = 1}}^\infty \frac{h(n)}{n} + O\left(  \frac{\varphi(Q)}{Q}  \log X \right) .
\end{equation}

Note that 
\begin{align*} \sum_{\substack{n \geq 1\\ \gcd(n,Q) = 1}} \frac{h(n)}{n} &= \prod_{\gcd(p,Q) = 1} \left(1 - \frac{1}{p^2}\right) 
= \prod_{p} \left(1 - \frac{1}{p^2}\right)   \prod_{p\mid Q} \left( 1 - \frac{1}{p^2} \right)^{-1}  \\
& = \frac{6}{\pi^2} \prod_{p\mid Q} \left( 1 - \frac{1}{p^2} \right)^{-1} = \frac{6}{\pi^2} \frac{Q}{\varphi(Q)}  \frac{1}{\psi(Q)}.  \end{align*}
Thus, we see from~\eqref{eq: G(Q,X) and h} that 
\[
G(Q,X) =\frac{6}{\pi^2 \psi(Q)} X+ O\left(   \frac{\varphi(Q)}{Q}\log X \right)
\]
and so by~\eqref{eq: G and F} we derive
\begin{equation}
\label{eq: F1} 
F_1(Q,X) =\frac{6}{\pi^2 \psi(Q)Q} X+O\left(  \frac{\varphi(Q)}{Q}  \log X \right).
\end{equation}

 \subsection{Approximating $F_2(Q,X)$}

Let 
\[
\delta_d(n) = \begin{cases} 1, & \text{if}\ d \mid n,\\
0, &  \text{if}\ d \nmid n,
\end{cases}
\]
be the characteristic function of the set of integer multiplies of an integer $d\ne 0$.
Then 
\begin{align*} \sum_{\substack{n \leq X\\ \gcd(n,Q)=1}} \frac{\varphi(n)}{n} &= \sum_{n \leq X} \prod_{p\mid Q} \(1 - \delta_p(n)\) \frac{\varphi(n)}{n} = \sum_{n \leq X} \sum_{d\mid Q} \mu(d)\delta_d(n) \frac{\varphi(n)}{n} \\
&= \sum_{d\mid Q} \mu(d) \sum_{n \leq X/d} \frac{\varphi(dn)}{dn}  
= \sum_{d\mid Q} \mu(d)  G(d, X/d)  . 
 \end{align*}

 We can now use~\eqref{eq: G(Q,X) and h} and the multiplicativity of  $\psi(d)$   to obtain 
 \begin{align*} \sum_{\substack{n \leq X\\ \gcd(n,Q)=1}} \frac{\varphi(n)}{n}  
&= \frac{6}{\pi^2}X \sum_{d\mid Q} \mu(d)\frac{1}{\psi(d)d} + O\(\sum_{d\mid Q} |\mu(d)| \frac{\varphi(d)}{d}  \log X \) \\
&= \frac{6}{\pi^2} X \prod_{p\mid Q} \left(1 - \frac{1}{\psi(p)p}\right) +  O\left( 2^{\omega(Q)} \log X \right) \\
 \end{align*}
 since 
 \[
 \sum_{d\mid Q} |\mu(d)| \frac{\varphi(d)}{d} \le \sum_{d\mid Q} |\mu(d)|  = 2^{\omega(Q)}, 
 \]
 where $\omega(Q)$ is the number of prime divisors of $Q$. 
 
 A simple computation shows that 
\[ \prod_{p\mid Q} \left(1 - \frac{1}{\psi(p)p}\right) =  \prod_{p\mid Q} \left(1 - \frac{1}{p+1}\right) =
\prod_{p\mid Q} \frac{1}{1+p^{-1}} =  \frac{1}{\psi(Q)}.
\]
Therefore 
   \begin{align*} \sum_{\substack{n \leq X\\ \gcd(n,Q)=1}} \frac{\varphi(n)}{n}   = \frac{6}{\pi^2} \frac{X}{\psi(Q)} + 
   O\(2^{\omega(Q)}  \log X \). 
 \end{align*}

Therefore,  using that $2^{\omega(Q)} \le \tau(Q) = Q^{o(1)}$   we obtain  
\begin{equation}
\label{eq: F2} 
\begin{split}
F_2(Q,X) & = Q^{-1} \frac{6}{\pi^2} \frac{X - Q}{\psi(Q)} + O\left(Q^{-1+o(1)} \log X \right)\\
& =    \frac{6}{\pi^2} \frac{X }{Q\psi(Q)} + O\(Q^{-1+o(1)} \log X \) , 
\end{split}
\end{equation}
wince $\psi(Q) \ge Q$.

\subsection{Concluding the proof} 
Combining the bounds~\eqref{eq: F1} and \eqref{eq: F2}  
we obtain the desired result.

  \section{Comments}  

We presented our result,  Corollary~\ref{cor:AsymFormula} as a direct consequence of
Theorems~\ref{thm:Card G}  and~\ref{thm:Func F} of very different nature with error terms of 
different strength. This makes it apparent that Theorem~\ref{thm:Card G} is the bottleneck 
to further improvements of Corollary~\ref{cor:AsymFormula}. 
 
The methods of this work can also be used for  counting elements of bounded norm of other congruence 
subgroup such as 
 \[
\Gamma(Q) = \left\{  \begin{bmatrix}
   a   & b \\
   c  & d \\
\end{bmatrix} \in\SL_2(\Z):~ a,d \equiv 1 \pmod Q, \ b,c \equiv 0 \pmod Q\right\}
\]
and 
 \[
\Gamma_1(Q) = \left\{  \begin{bmatrix}
   a   & b \\
   c  & d \\
\end{bmatrix} \in\SL_2(\Z):~ a,d\equiv 1 \pmod Q, \ c  \equiv 0 \pmod Q\right\}.
\]

One can also adjust our approach to counting matrices of restricted  size with respect to
other natural matrix norms. 
 


\section*{Acknowledgements}  

The authors would like to thank R{\'e}gis de la Bret{\'e}che for his suggestion 
which has led to the proof of Theorem~\ref{thm:Func F} with the current error term. 

During the preparation of this work, the authors were   supported in part by the  
Australian Research Council Grants DP230100530 and DP230100534.

\end{document}